\documentclass[11pt,a4paper]{article}
\usepackage[margin=27mm]{geometry}
\usepackage{amsmath,amssymb,amsthm}
\usepackage{microtype}
\usepackage[hidelinks]{hyperref}
\allowdisplaybreaks[2]
\newtheorem{theorem}{Theorem}[section]
\newtheorem{lemma}[theorem]{Lemma}
\newtheorem{proposition}[theorem]{Proposition}
\newtheorem{corollary}[theorem]{Corollary}
\theoremstyle{remark}
\newtheorem{remark}[theorem]{Remark}
\newtheorem*{remark*}{Remark}
\newtheorem{example}[theorem]{Example}
\theoremstyle{plain}
\newcommand{\F}{\mathbb F}
\hypersetup{
  pdftitle={Covering finite fields by translates of multiplicative subgroups},
  pdfauthor={Juncheng Zhou and Hongfeng Wu}}
\title{Covering finite fields by translates of multiplicative subgroups}
\author{Juncheng Zhou and Hongfeng Wu\thanks{Corresponding author: Hongfeng Wu.
E-mail addresses: jczhoumath@gmail.com (J. Zhou), whfmath@gmail.com (H. Wu).}\\[3pt]
\small College of Science, North China University of Technology, Beijing, China}
\date{}
\begin{document}
\maketitle

\begin{abstract}
Let $G_d$ be the subgroup of index $d>1$ in $\F_q^\times$, and put
$G_d^0=G_d\cup\{0\}$. For a covering $A+G_d^0=\F_q$ with $\#A=s$,
we prove $q\le d^{8s/d}$ by shifting the roots of a Bonferroni polynomial.
Although a uniform bound in $s$ alone must allow exponential growth,
our estimate confines such growth to bounded indices and gives
polynomial bounds when $s/d$ is bounded. A cubic specialization gives
bounds of degree at most eight for $s\le2d$.
For shift sets of logarithmic size, higher moment estimates give a
uniform covering-probability formula, including the critical window.
Together with the index restriction, this determines the covering
indices of almost every shift set of size $\lceil k\log q\rceil$
for each fixed integer $k\ge2$.

A subfield construction attains the order of the uniform restriction
on the index. We determine its covering number exactly and prove
minimality under subgroup inclusion. We also characterize
equality in the elementary counting bound by subfield cosets.
Fixing a subfield as the shift set in a quadratic extension leads
to a sharper projection bound, attained by an infinite family with
unique representations outside the shift subfield.
\end{abstract}

\section{Introduction}
Let $\F_q$ be the finite field of order $q$. For a positive divisor
$d$ of $q-1$, write
\[
 G_d=(\F_q^\times)^d,\qquad G_d^0=G_d\cup\{0\}.
\]
We consider coverings
\begin{equation}\label{eq:covering}
 A+G_d^0=\F_q,\qquad A\subseteq\F_q,\quad \#A=s\ge1,\quad d>1.
\end{equation}
We ask how the field size depends on the number of shifts and the index.
Since $\#G_d^0=(q-1)/d+1$, counting gives $d\le2s-1$, but no bound on
$q$ for fixed $s$. Character sums provide such a bound. Our aim is to
make its dependence on $d$ explicit and to understand how the shift
set affects the possible indices.
All logarithms are natural. For integers $n\ge0$, we set
$\binom nk=0$ when $k<0$ or $k>n$.

Zhou, Danchev and Wu \cite[Theorem~2.1]{five} proved
that \eqref{eq:covering} implies
\begin{equation}\label{eq:uniform}
 q\le(s2^{s-1})^2.
\end{equation}
They also showed, by inclusion--exclusion through order three, that
\begin{equation}\label{eq:previous-large-index}
 \frac{s}{d}\le\frac32\quad\Longrightarrow\quad q\le256s^6
\end{equation}
under the same hypothesis \cite[Theorem~3.1]{five}.
The exponential dependence on $s$ in \eqref{eq:uniform} is unavoidable.
For every odd $q$, averaging gives a covering by translates of $G_2^0$
with at most $\lceil\log q/\log2\rceil$ shifts; see the proof of
Proposition~\ref{prop:covering-number}. Enlarging the shift set if necessary,
we may take $s=\lceil\log q/\log2\rceil$, so $q>2^{s-1}$.
Hence no uniform subexponential bound in $s$ is possible, although
the optimal exponential rate in \eqref{eq:uniform} remains undetermined.

We therefore ask how the index restricts the field size as $d$ grows.
Shifting the roots of the standard Bonferroni polynomial and choosing
the degree in terms of $s/d$ gives our main bound.

\begin{theorem}\label{thm:index-bound}
Suppose that \eqref{eq:covering} holds. Then
\begin{equation}\label{eq:index-bound}
 q\le d^{8s/d}.
\end{equation}
\end{theorem}

Taking logarithms gives
\[
 \log q\le8s\frac{\log d}{d}.
\]
Thus $\log q=o(s)$ along every sequence of coverings with
$d\to\infty$, with no restriction on $s/d$. Exponential growth
$q\ge\exp(cs)$ for a fixed $c>0$ is therefore possible only at bounded
index, where the bound need not improve \eqref{eq:uniform}.

When $s/d$ stays bounded, Theorem~\ref{thm:index-bound} gives a
polynomial bound. The ratio $s/d$ measures the average representation
multiplicity: averaged over $x\in\F_q$, the number of representations
$x=a+h$ with $a\in A$ and $h\in G_d$ is $(s/d)(1-1/q)$.
Since $d<2s$, for each fixed $C>0$,
\begin{equation}\label{eq:bounded-ratio}
 \frac{s}{d}\le C\quad\Longrightarrow\quad q=O_C(s^{8C}).
\end{equation}
This extends the range of polynomial bounds from $s/d\le3/2$ in
\eqref{eq:previous-large-index} to every bounded range of $s/d$.

For small $s/d$, a cubic specialization gives sharper estimates
(Section~\ref{sec:polynomial-tests}): $q\le4s^8$ for $s\le2d$, and
$q\le4\varepsilon^{-2}s^6$ for $s\le(2-\varepsilon)d$, where
$0<\varepsilon\le1$. The sixth-power estimate extends the range of
\eqref{eq:previous-large-index} to every fixed ratio $s/d<2$;
on the original range $s/d\le3/2$, it improves the constant from $256$ to $16$.
At $s=d$, a direct intersection argument gives the further improvement
$q<d^4$.

These estimates also describe how the index affects the number of shifts
needed for a covering. For fixed $d$, \eqref{eq:uniform} and averaging
already show that the minimum is of order $\log q$.
Writing $\tau(q,d)$ for this minimum, Proposition~\ref{prop:covering-number}
makes the dependence on $d$ explicit:
\[
 \frac{d\log q}{8\log d}\le\tau(q,d)\le\lceil d\log q\rceil.
\]
Inverting the lower estimate gives the uniform index bound in
Corollary~\ref{cor:index-growth}:
\[
 d\ll\frac{s}{\log q}\log\left(2+\frac{s}{\log q}\right).
\]
In particular, $s=O(\log q)$ forces $d=O(1)$.

Higher moments also give a uniform estimate for the proportion of
covering shift sets at logarithmic sizes. For each fixed $C>0$, if $A$
is chosen uniformly among the $s$-element subsets of $\F_q$, then
\[
 \Pr(A+G_d^0=\F_q)
 =\exp\bigl(-q(1-1/d)^s\bigr)+O_C(q^{-1/12}),
\]
uniformly for $1\le s\le C\log q$ and all $d>1$ dividing $q-1$
(Proposition~\ref{prop:uniform-covering}). This includes the critical window.
Combining the formula with the uniform index restriction shows that,
for each fixed integer $k\ge2$, almost every set of
$\lceil k\log q\rceil$ shifts has precisely the
covering indices $d\mid q-1$ with $2\le d\le k$
(Corollary~\ref{cor:counting-threshold}).

A subfield construction with $q\asymp s^2$ and $d\asymp s$
attains the order of the uniform restriction on the index
(Example~\ref{prop:quadratic-covering}). In this family, the subfield
is an optimal shift set, and replacing either multiplicative subgroup
by a proper subgroup destroys the covering while the other summand
is fixed.

We examine this subfield structure further in Section~\ref{sec:subfields}.
Equality in the elementary counting bound holds precisely when
$G_d^0$ is a subfield and $A$ is a complete set of its additive coset
representatives. With a subfield fixed as the shift set, projection
onto the additive quotient gives a sharper restriction: a covering
with shift set $\F_q\subset\F_{q^2}$ requires
\[
 d+\gcd(d,q+1)\le q+1.
\]
Cube roots of unity give an infinite family attaining equality in
this projection bound, with unique representations outside the
shift subfield (Example~\ref{prop:cubic-family}).

\section{An index-dependent upper bound}\label{sec:index-proof}
We first estimate intersections of translates and polynomial moments
of their representation counts. A shifted Bonferroni polynomial then
gives the lower bound needed for Theorem~\ref{thm:index-bound}.
The two lemmas below will also be used in the later sections.

For $d>1$ dividing $q-1$, write $\alpha_d=1-1/d$.
We begin with the intersection estimate; see also
\cite[Lemma~1.1]{martinyip}.
\begin{lemma}\label{lem:intersection}
Let $d>1$ divide $q-1$, and let $X\subseteq\F_q$ have cardinality
$k\ge1$. Put
\[
 I_d(X)=\bigcap_{x\in X}(x+G_d).
\]
Then
\begin{equation}\label{eq:intersection-explicit}
 \left|\#I_d(X)-\frac q{d^k}\right|
 \le\left(k\alpha_d-1+d^{-k}\right)\sqrt q+\frac kd,
\end{equation}
or, more crudely,
\[
 \left|\#I_d(X)-\frac q{d^k}\right|\le k\sqrt q.
\]
Both estimates hold for every $d,k,X$ as above, including when $k$
grows with $q$. For $k=2$, the coefficient of
$\sqrt q$ in \eqref{eq:intersection-explicit} is $\alpha_d^2$.
\end{lemma}
\begin{proof}
Write $X=\{x_1,\ldots,x_k\}$ and choose a character $\chi$ of order $d$.
We use $\chi^0=1$ and extend every nontrivial power of $\chi$ by zero
at the origin. Away from $X$, the indicator of $G_d$ is
$d^{-1}\sum_{j=0}^{d-1}\chi^j$, so
\[
 d^k\#I_d(X)
 =\sum_{j_1,\ldots,j_k=0}^{d-1}
   \sum_{x\in\F_q\setminus X}
   \prod_{i=1}^k\chi^{j_i}(x-x_i).
\]
For an exponent vector with $h\ge1$ nonzero coordinates,
the Weil bound gives
\begin{equation}\label{eq:weil}
 \left|\sum_{x\in\F_q}\prod_{i=1}^k\chi^{j_i}(x-x_i)\right|
 \le(h-1)\sqrt q;
\end{equation}
see \cite{schmidt}. Indeed, the corresponding polynomial has $h$
distinct roots, each of multiplicity strictly between $0$ and $d$,
so it is not a constant multiple of a $d$th power.
The summand vanishes at these $h$ roots, so removing the other
$k-h$ points of $X$ changes the sum by at most $k-h$.
The zero vector contributes $q-k$.
Across the $d^k$ vectors, each coordinate is nonzero $(d-1)d^{k-1}$
times and zero $d^{k-1}$ times. Thus the total error from omitting
$X$, including the zero vector, is at most $kd^{k-1}$, while the
sum of $h-1$ over the $d^k-1$ nonzero vectors is
$k(d-1)d^{k-1}-d^k+1$. Dividing by $d^k$ gives
\[
 \begin{aligned}
 \left|\#I_d(X)-\frac q{d^k}\right|
 &\le\frac{k(d-1)d^{k-1}-d^k+1}{d^k}\sqrt q+\frac kd\\
 &=\left(k\alpha_d-1+d^{-k}\right)\sqrt q+\frac kd.
 \end{aligned}
\]
Since $k/d\le(k/d)\sqrt q$ and $d^{-k}\le1$, this is at most
$k\sqrt q$.
\end{proof}

For $A\subseteq\F_q$ with $\#A=s\ge1$, put
\[
 r_A(x)=\#\{a\in A:x-a\in G_d\}.
\]

\begin{lemma}\label{lem:bonferroni}
Let $A\subseteq\F_q$ with $\#A=s\ge1$, and let $d>1$ divide $q-1$.
Let $\ell\ge1$ be an integer and
\[
 P(t)=\sum_{j=0}^\ell c_j\binom tj,\qquad c_j\in\mathbb R.
\]
Then
\begin{equation}\label{eq:polynomial-moment}
 \left|\sum_{x\in\F_q}P(r_A(x))
       -q\sum_{j=0}^\ell c_j\binom sjd^{-j}\right|
 \le\sqrt q\sum_{j=1}^\ell j|c_j|\binom sj.
\end{equation}
\end{lemma}
\begin{proof}
For $1\le j\le s$, $\binom{r_A(x)}j$ counts the choices of $j$ translates
$a+G_d$ containing $x$. Summing over the points and then counting by
the chosen shifts gives
\[
 \sum_{x\in\F_q}\binom{r_A(x)}j
 =\sum_{\substack{J\subseteq A\\\#J=j}}
   \#I_d(J).
\]
Applying Lemma~\ref{lem:intersection} to each $I_d(J)$ gives
\[
 \left|\sum_{x\in\F_q}\binom{r_A(x)}j-q\binom sjd^{-j}\right|
 \le j\binom sj\sqrt q.
\]
The term with $j=0$ is exact, and both terms vanish for $j>s$.
Expanding $P$ and applying the triangle inequality proves
\eqref{eq:polynomial-moment}.
\end{proof}

Write
\[
 S_d(A)=\#\{x\in\F_q:r_A(x)=0\}
\]
for the number of points outside $A+G_d$. Since
$A+G_d^0=(A+G_d)\cup A$, condition~\eqref{eq:covering} implies
$S_d(A)\le s$. The next consequence is a Bonferroni-type lower bound;
for a general treatment of the polynomial method, see
\cite[Theorem~2.1]{sibuya}.

\begin{corollary}\label{cor:bonferroni-lower}
Let $A,d,P$ be as in Lemma~\ref{lem:bonferroni}.
If $c_0>0$ and $P(n)\le0$ for every integer $1\le n\le s$, then
\begin{equation}\label{eq:bonferroni-lower}
 q\sum_{j=0}^\ell c_j\binom sjd^{-j}
 \le c_0S_d(A)+\sqrt q\sum_{j=1}^\ell j|c_j|\binom sj.
\end{equation}
Under \eqref{eq:covering}, one may replace $S_d(A)$ on the right by $s$.
\end{corollary}
\begin{proof}
The sign conditions give $\sum_xP(r_A(x))\le c_0S_d(A)$.
The assertions follow from \eqref{eq:polynomial-moment} and
$S_d(A)\le s$ under the covering condition.
\end{proof}

For a positive odd integer $\ell$, a standard choice in
Corollary~\ref{cor:bonferroni-lower} is the Bonferroni polynomial
\[
 \prod_{j=1}^\ell\left(1-\frac tj\right)
 =\sum_{j=0}^\ell(-1)^j\binom tj.
\]
It equals $1$ at $t=0$ and is nonpositive at every positive integer;
see \cite[Chapter~2, Lemma~2.1]{halberstamrichert}.
For Theorem~\ref{thm:index-bound}, we keep the root $1$ and shift the
remaining roots one unit to the right.
For odd $\ell\ge3$, this gives
\begin{equation}\label{eq:shifted-weight}
\begin{aligned}
 P_\ell(t)&=\prod_{\substack{1\le j\le\ell+1\\j\ne2}}
             \left(1-\frac tj\right)
           =\sum_{j=0}^\ell c_j\binom tj,\\
 c_j&=(-1)^j\left(1-\frac{j(j-1)}{\ell(\ell+1)}\right).
\end{aligned}
\end{equation}
Among the integers $0\le t\le\ell+1$, the only nonzero values are
$P_\ell(0)=1$ and $P_\ell(2)=-2/(\ell(\ell+1))$.
Binomial inversion therefore gives the displayed formula for $c_j$.
The roots after $1$ still occur in adjacent integer pairs, so the sign
condition is preserved; see also \cite[Theorem~7.1(2)]{sibuya}.
Write the main and error coefficients as
\[
 T_\ell=\sum_{j=0}^\ell c_j\binom sjd^{-j},\qquad
 E_\ell=\sum_{j=1}^\ell j|c_j|\binom sj.
\]
Since $c_0=1$, \eqref{eq:bonferroni-lower} gives, under \eqref{eq:covering},
\begin{equation}\label{eq:bonferroni-covering}
 qT_\ell\le s+E_\ell\sqrt q.
\end{equation}

The shift introduces a loss at $t=2$ but delays the tail to $t=\ell+2$.
Estimating these two losses separately allows us to take a degree just
below $4s/d$ while retaining a fixed positive proportion of
$\alpha_d^s$ in $T_\ell$. Keeping the actual coefficients in the error
estimate gives the same factor, which cancels in the final comparison.

\begin{proof}[Proof of Theorem~\ref{thm:index-bound}]
Counting gives $d<2s$. If $s<d$, it also gives
\[
 q\le\frac{s(d-1)}{d-s}\le s^2<d^{8s/d}.
\]
For $2\le d\le9$, apply $\log x\le x/\mathrm e$ with $x=s/2$
in \eqref{eq:uniform} to obtain
\[
 q\le s^2 4^{s-1}
 \le\exp\bigl((\log4+1/\mathrm e)s\bigr)\le d^{8s/d},
\]
since $\log4+1/\mathrm e<\min_{2\le d\le9}8\log d/d$.

We may therefore assume $d\ge10$ and put $\mu=s/d\ge1$.
Take $\ell$ to be the least odd integer
with $\ell\ge4\mu-2$. Then $3\le\ell<4\mu\le\ell+2<s$.
Using $\binom st\binom tj=\binom sj\binom{s-j}{t-j}$ and the
binomial theorem in the expansion \eqref{eq:shifted-weight}, we obtain
\[
 T_\ell=\sum_{t=0}^s P_\ell(t)\binom st d^{-t}\alpha_d^{s-t}.
\]
The terms with $t=0$ and $t=2$ are the only contributions below
$\ell+2$. For integers $t\ge\ell+2$,
\[
 -P_\ell(t)=\frac{2(\ell+2)}{t(t-2)}\binom t{\ell+2}
 \le\frac2\ell\binom t{\ell+2}.
\]
Summing these contributions, using the same identity for the tail, gives
\begin{equation}\label{eq:shifted-main-bound}
 T_\ell\ge\alpha_d^s
 \left(1-\frac{s(s-1)}{\ell(\ell+1)(d-1)^2}\right)
 -\frac2\ell\binom s{\ell+2}d^{-\ell-2}.
\end{equation}
Since $\mu\le(\ell+2)/4$ and $\ell\ge3$, the first loss is at most
\[
 \frac{s(s-1)}{\ell(\ell+1)(d-1)^2}
 \le\left(\frac d{d-1}\right)^2
       \frac{(\ell+2)^2}{16\ell(\ell+1)}
 \le\frac{25}{192}\left(\frac{10}{9}\right)^2<\frac16.
\]
For the tail, we use the comparison
\[
 (\mathrm e/4)^{4\mu}\le(5/8)^\mu\alpha_d^s,
\]
which follows from
$(\mathrm e/4)^4<(5/8)(9/10)^{10}\le(5/8)\alpha_d^d$,
since $\alpha_d^d$ increases with $d$.
Stirling's bound $k!\ge\sqrt{2\pi k}(k/\mathrm e)^k$ and
$\ell+2\ge4\mu$ now give
\[
 \frac2\ell\binom s{\ell+2}d^{-\ell-2}
 \le\frac{2(\mathrm e/4)^{4\mu}}{\ell\sqrt{2\pi(\ell+2)}}
 \le\frac1{12}\alpha_d^s,
\]
where the last step uses $\mu\ge1$, $\ell\ge3$, and
$5/(12\sqrt{10\pi})<1/12$.
Combining the two losses in \eqref{eq:shifted-main-bound}, we obtain
\begin{equation}\label{eq:positive-truncation}
 T_\ell\ge\tfrac34\alpha_d^s.
\end{equation}

For the error, \eqref{eq:shifted-weight} gives
\[
 j|c_j|=\frac{j(\ell+j)(\ell-j+1)}{\ell(\ell+1)}
 \le2(\ell-j+1)\qquad(1\le j\le\ell).
\]
Using $\binom sj\le s^j/j!$ and summing the weighted geometric series in
$h=\ell-j$ gives
\[
 E_\ell\le\frac{2s^\ell}{\ell!(1-\ell/s)^2}
 \le\frac{50}{9\sqrt{2\pi\ell}}(\mathrm es/\ell)^\ell,
\]
where we used $\ell/s<2/5$ and Stirling's bound.
Since $(4\mu/\ell)^\ell\le\exp(4\mu-\ell)$, the preceding comparison gives
\[
 \frac{E_\ell}{d^{4\mu}\alpha_d^s}
 \le\frac{50}{9\sqrt{2\pi\ell}}(5/8)^\mu(4/d)^{4\mu-\ell}
 \le\frac9{16}.
\]
For the last inequality, if $\ell=3$, use $\mu\ge1$ and
$4\mu-\ell\ge1$; if $\ell\ge5$, use $\mu>5/4$ and
$50(5/8)^{5/4}/(9\sqrt{10\pi})<9/16$.
Finally, $q\ge d+1>9$ and $E_\ell\ge s$, so
$s+E_\ell\sqrt q\le\tfrac43E_\ell\sqrt q$.
Combining \eqref{eq:bonferroni-covering} with the bounds for $T_\ell$
and $E_\ell$ gives
\[
 \sqrt q\le\frac{4E_\ell}{3T_\ell}\le d^{4\mu}.
\]
Squaring proves the theorem.
\end{proof}
 
\section{Sharper polynomial bounds}\label{sec:polynomial-tests}
The preceding proof chooses the degree according to $s/d$.
For $s\le2d$, we use the fixed choice $\ell=3$ in
\eqref{eq:shifted-weight}:
\[
 P_3(t)=\prod_{j\in\{1,3,4\}}\left(1-\frac tj\right),
\]
which also occurs in
\cite[p.~273, proof of Proposition~4.1(2-2)]{sibuya}.
Estimating its main coefficient $T_3$ directly gives a sixth-power bound
when $s/d$ stays a fixed distance below $2$, and an eighth-power bound
up to the endpoint. The small cases are handled separately by the
counting bound and \eqref{eq:uniform}.

\begin{theorem}\label{thm:cubic}
Suppose that \eqref{eq:covering} holds. For each $0<\varepsilon\le1$,
\begin{equation}\label{eq:sixth}
 s\le(2-\varepsilon)d\quad\Longrightarrow\quad
 q\le4\varepsilon^{-2}s^6.
\end{equation}
Moreover,
\begin{equation}\label{eq:eighth}
 s\le2d\quad\Longrightarrow\quad q\le4s^8.
\end{equation}
\end{theorem}
\begin{proof}
We need only consider $s\le2d$. Every covering has $s\ge2$. If $s<d$, the elementary bound
$q\le s(d-1)/(d-s)\le s^2$ proves both assertions.
If $s\le6$, then \eqref{eq:uniform} gives $q\le s^6$, which also suffices.
We may therefore assume $d\le s\le2d$ and $s\ge7$, so $d\ge4$.

For the cubic weight, the coefficient of $q$ in
\eqref{eq:bonferroni-covering} is
\[
 T_3=1-\frac sd+\frac56\binom s2d^{-2}-\frac12\binom s3d^{-3}.
\]
To make its dependence on $s/d$ and $1/d$ explicit, put
\[
 \mu=\frac sd,\qquad \theta=\frac1d.
\]
Here $1\le\mu\le2$ and $0<\theta\le1/4$.
Then $T_3=\mathcal P_3(\mu,\theta)$, where
\begin{equation}\label{eq:cubic-mean}
 \mathcal P_3(\mu,\theta)
 :=1-\mu+\frac5{12}\mu(\mu-\theta)
       -\frac1{12}\mu(\mu-\theta)(\mu-2\theta).
\end{equation}
The part independent of $\theta$ factors as
\[
 \mathcal P_3(\mu,0)=\frac{(2-\mu)(\mu^2-3\mu+6)}{12}.
\]
This is positive for $\mu<2$ but vanishes at $\mu=2$, so the terms
involving $\theta$ are needed at the endpoint.
Completing the square in the derivative gives, for $0<\theta\le1/4$,
\[
 \frac{\partial\mathcal P_3}{\partial\mu}
 =-\frac14\left(\mu-\theta-\frac53\right)^2
   +\frac{\theta^2+5\theta-11/3}{12}\le-\frac3{16}.
\]
At the endpoint,
\[
 \mathcal P_3(2,\theta)=\frac{\theta(1-2\theta)}6\ge\frac\theta{12}.
\]
Integrating the derivative bound from $\mu$ to $2$, with $\theta$ fixed,
and then using $\theta=\mu/s$, we obtain
\begin{equation}\label{eq:cubic-positive}
\begin{aligned}
 \mathcal P_3(\mu,\theta)
 &\ge\frac{\theta}{12}+\frac3{16}(2-\mu)\\
 &=\frac{2-\mu+s^{-1}}6+\frac{(s-4)(2-\mu)}{48s}\\
 &\ge\frac{2-\mu+s^{-1}}6.
\end{aligned}
\end{equation}
The remainder is nonnegative since $s\ge7$ and $\mu\le2$.

The error coefficient $E_3=(3s^3+s^2+8s)/12$ satisfies
$s+E_3\le s^3/3$ for $s\ge7$.
Thus \eqref{eq:bonferroni-covering} and $\sqrt q\ge1$ give
\[
 q\mathcal P_3(\mu,\theta)\le\frac{s^3}{3}\sqrt q.
\]
Combining this with \eqref{eq:cubic-positive} and squaring gives
\begin{equation}\label{eq:cubic-finite}
 q\le\frac{4s^6}{(2-s/d+s^{-1})^2}.
\end{equation}
The initial elementary estimates give the same bound when $s<d$ or
$s\le6$. Finally, $2-s/d+s^{-1}\ge\varepsilon$ if
$s\le(2-\varepsilon)d$, and it is at least $s^{-1}$ if $s\le2d$.
This proves both assertions.
\end{proof}

\begin{remark}\label{rem:cubic-range}
The sixth-power bound extends the range $s\le3d/2$ of
\eqref{eq:previous-large-index} to every fixed ratio $s/d<2$.
On the original range, taking $\varepsilon=1/2$ gives $q\le16s^6$,
improving the earlier constant $256$.
At $s=2d$, the coefficient $\mathcal P_3(2,1/d)$ is of order $1/d$
for sufficiently large $d$. For this test, dividing by the coefficient and then
squaring introduces the two additional powers of $s$ in \eqref{eq:eighth}.
\end{remark}

At $s=d$, only $d-1$ representations remain after assigning one to
each point. The two-translate estimate then gives a fourth-power bound.

\begin{proposition}\label{prop:critical-line}
If \eqref{eq:covering} holds with $s=d$, then $q<d^4$.
\end{proposition}
\begin{proof}
Write
\[
 r_A^0(x)=\#\{a\in A:x-a\in G_d^0\}
\]
for the representation count including zero.
Since $r_A^0(x)\ge1$ and $s=d$, we have
\[
 \sum_{x\in\F_q}(r_A^0(x)-1)=d-1.
\]
Choose distinct $a,a'\in A$ and put $t=a'-a$.
Each point common to these two translates contributes at least one
to this excess. Hence Lemma~\ref{lem:intersection} and
$G_d\subseteq G_d^0$ give
\[
 \frac q{d^2}-\frac{(d-1)^2}{d^2}\sqrt q-\frac2d
 \le\#(G_d^0\cap(G_d^0+t))\le d-1.
\]
Writing $y=\sqrt q$ and rearranging gives
\[
 y^2-(d-1)^2y-d^3+d^2-2d\le0.
\]
The quadratic has a negative constant term and hence exactly one
positive root. Its value at $y=d^2$ is $d(d^2-2)>0$, so $y<d^2$
and $q<d^4$.
\end{proof}

\section{Consequences for additive coverings}\label{sec:consequences}
We first bound the covering number and the index for every shift set,
and give an example attaining the order of the index restriction.
We then compare covering counts and derive a uniform covering-probability
formula for logarithmic sizes, leading to the simultaneous index
classification in Corollary~\ref{cor:counting-threshold}.

\subsection{Covering numbers and index growth}
For $X\subseteq\F_q$, write
\[
 U_d(X)=\#\bigl(\F_q\setminus(X+G_d^0)\bigr)
\]
for the number of uncovered points. The minimum number of translates
needed to cover the field is
\[
 \tau(q,d)=\min\{\#A:U_d(A)=0\}.
\]
The following bounds combine Theorem~\ref{thm:index-bound} with the
standard averaging argument for translates \cite[Section~3]{bjr}.

\begin{proposition}\label{prop:covering-number}
For every $d>1$ dividing $q-1$,
\[
 \frac{d\log q}{8\log d}\le\tau(q,d)\le\lceil d\log q\rceil.
\]
\end{proposition}
\begin{proof}
Taking logarithms in Theorem~\ref{thm:index-bound} with
$s=\tau(q,d)$ gives the lower bound.

For the upper bound, suppose that $U_d(A)>0$. Each uncovered point
remains uncovered after adding $a+G_d^0$ for exactly $q-\#G_d^0$
choices of $a$. Hence
\[
 \sum_{a\in\F_q}U_d(A\cup\{a\})=(q-\#G_d^0)U_d(A).
\]
Since $\#G_d^0>q/d$, averaging gives a shift that leaves fewer than
$(1-1/d)U_d(A)$ points uncovered. Starting with $A=\varnothing$,
repeat this choice while $U_d(A)>0$. After $t\ge1$ steps, either the
field is already covered or
\[
 U_d(A)<q(1-1/d)^t.
\]
Since $U_d(A)$ is a nonnegative integer, the field is covered once
$q(1-1/d)^t\le1$. Thus
\[
 \tau(q,d)\le\left\lceil\frac{\log q}{\log(d/(d-1))}\right\rceil
 \le\lceil d\log q\rceil.
\]
The last inequality uses $\log(d/(d-1))\ge1/d$.
\end{proof}

\begin{remark}\label{rem:covering-number}
The elementary count gives a lower bound of order $d$. The additional
factor $\log q/\log d$ in Proposition~\ref{prop:covering-number} is significant
when $\log d=o(\log q)$. The lower and upper bounds differ by a factor
of order $\log d$; for fixed $d$, they give
$\tau(q,d)\asymp_d\log q$.
\end{remark}

Inverting the lower bound for $\tau(q,d)$ restricts the index in terms
of $s/\log q$.

\begin{corollary}\label{cor:index-growth}
Suppose that \eqref{eq:covering} holds. Then
\[
 d\ll\frac{s}{\log q}\log\left(2+\frac{s}{\log q}\right),
\]
with an absolute implied constant.
\end{corollary}
\begin{proof}
Put $\Lambda=8s/\log q$. Proposition~\ref{prop:covering-number} gives
$d\le \Lambda\log d$ and $\Lambda\ge d/\log d\ge\mathrm e$.
Since $\log t\le t/2$ for $t>0$,
\[
 \frac d\Lambda\le\log \Lambda+\log\frac d\Lambda
 \le\log \Lambda+\frac d{2\Lambda}.
\]
Thus $d\le2\Lambda\log \Lambda$, which proves the corollary.
\end{proof}

\begin{remark*}
Corollary~\ref{cor:index-growth} bounds $d$ uniformly over all shift
sets when $s=O(\log q)$, and forces $s/\log q\to\infty$ when
$d\to\infty$. Together with the averaging upper bound, this gives
\[
 \tau(q,d)=O(\log q)\quad\Longleftrightarrow\quad d=O(1)
\]
along sequences with $q\to\infty$.
\end{remark*}

The following example attains the order of the index bound in
Corollary~\ref{cor:index-growth}.

\begin{example}\label{prop:quadratic-covering}
Take an odd prime power $q$ and put
\[
 C_m=\{x\in\F_{q^2}:x^m=x\}\qquad(m=q,2q-1).
\]
Choose $u\in\F_{q^2}^\times$ of order $2(q-1)$. Then $u^2$ generates
$\F_q^\times$ and $u\notin\F_q$, so
\[
 C_q=\F_q,\qquad C_{2q-1}=\F_q\cup u\F_q.
\]
In the basis $1,u$, $C_{2q-1}$ is the union of the coordinate axes.
The $q$ shifts from $\F_q$ cover $\F_{q^2}$. With fewer than $q$
centers, their two coordinate projections each omit a value, leaving
the point with these coordinates uncovered. Since $C_{2q-1}\setminus\{0\}$
has index $(q+1)/2$, we obtain
\[
 C_q+C_{2q-1}=\F_{q^2},\qquad
 \tau\!\left(q^2,\frac{q+1}{2}\right)=q.
\]

Neither nonzero subgroup can be replaced by a proper subgroup while
the other summand is fixed. Indeed, if $J$ is a proper subgroup of
$C_{2q-1}\setminus\{0\}$, then $\#J\le q-1$. A covering
$\F_q+(J\cup\{0\})=\F_{q^2}$ forces $\#J=q-1$ by counting,
and hence $J=\F_q^\times$, but then the sum is $\F_q$.
Any proper subgroup of $\F_q^\times$, with zero adjoined, has fewer
than $q$ elements and is ruled out by the covering number.

The restriction to subgroups is essential: the smaller set
$u\F_q\subsetneq C_{2q-1}$ still satisfies $\F_q+u\F_q=\F_{q^2}$,
but its nonzero part is not a multiplicative subgroup. This minimality
excludes proper multiples of the index, without asserting that the
index is numerically largest.

For the shift set $A=\F_q$, we have $s=q$ and $d=(q+1)/2$.
Since the ambient field has $q^2=s^2$ elements,
\[
 \frac{s}{2\log q}\log\left(2+\frac{s}{2\log q}\right)\sim d
 \qquad(q\to\infty).
\]
The logarithmic factor cannot be omitted, since $s/(2\log q)=o(d)$.
\end{example}

\subsection{Counting shift sets}
We now count covering shift sets of a prescribed size. Fix an integer
$s$ with $1\le s\le q$ and write $N=\binom qs$. All sums over $A$ in
this subsection range over the $s$-element subsets of $\F_q$.

For $d>1$ dividing $q-1$, put
\[
 n_d=U_d(\{0\})=q-\#G_d^0=(q-1)\alpha_d.
\]
Count pairs consisting of an $s$-element shift set and an $r$-element
subset of its uncovered points. The bijection
$(A,X)\mapsto(-X,-A)$ preserves the avoidance condition, since
$(-a)-(-x)=x-a$, and interchanges the two sizes. Thus, for $1\le r\le q$,
\begin{equation}\label{eq:uncovered-double-count}
 \sum_A\binom{U_d(A)}r
 =\sum_{\substack{X\subseteq\F_q\\\#X=r}}\binom{U_d(X)}s.
\end{equation}
Taking $r=1$ gives the first-moment identity
\[
 \sum_A U_d(A)=q\binom{n_d}s.
\]
Thus the average number of uncovered points is nondecreasing in $d$.
Let $N_d$ be the number of $s$-element shift sets that fail to cover
with $G_d^0$. Since $1\le U_d(A)\le q$ on these sets,
\[
 \binom{n_d}s\le N_d\le q\binom{n_d}s.
\]
In particular, $N_d/N\le q\alpha_d^s$. These bounds hold for every $s$.

To estimate higher moments, we first bound $U_d(X)$ for a $k$-element
set $X\subseteq\F_q$, $k\ge1$. Apply Lemma~\ref{lem:bonferroni}
with $A=X$ and the complete Bonferroni polynomial
\[
 P(t)=\sum_{j=0}^k(-1)^j\binom tj.
\]
On the integers $0\le t\le k$, this is the indicator of $t=0$, so
$\sum_xP(r_X(x))=S_d(X)$. Its main term in
\eqref{eq:polynomial-moment} is $q\alpha_d^k$, and
$\sum_{j=1}^k j\binom kj=k2^{k-1}$. Since
$0\le S_d(X)-U_d(X)\le k$, we obtain
\begin{equation}\label{eq:uncovered-centers}
 \left|U_d(X)-q\alpha_d^k\right|
 \le k2^{k-1}\sqrt q+k.
\end{equation}
This estimate is uniform in $d,k,X$ and allows $k$ to grow with $q$.

For $s=O(\log q)$ and sufficiently large $q$, we estimate the first two moments
uniformly in $d$. The product formula gives
\[
 \frac1N\sum_A U_d(A)
 =q\alpha_d^s\bigl(1+O(s^2/q)\bigr).
\]
For every two-element set $X$, \eqref{eq:uncovered-centers} gives
$U_d(X)=q\alpha_d^2+O(\sqrt q)$. Taking $r=2$ in
\eqref{eq:uncovered-double-count} and applying the product formula
again, with $\alpha_d\ge1/2$, gives
\[
 \frac1N\sum_A U_d(A)(U_d(A)-1)
 =q(q-1)\alpha_d^{2s}\bigl(1+O(s/\sqrt q)\bigr).
\]
Only $N_d$ terms are nonzero, so Cauchy--Schwarz and the two
moment estimates give
\[
 \frac{N_d}{N}
 \ge\frac{\left(\sum_A U_d(A)\right)^2}{N\sum_A U_d(A)^2}
 =\frac{q\alpha_d^s}{1+q\alpha_d^s}
   \bigl(1+O(s/\sqrt q)\bigr).
\]
Together with the exact bounds on $N_d$, this gives the following
comparison of covering counts.

\begin{proposition}\label{prop:min-index-comparison}
Fix a prime $d_0$, and let $q$ be a sufficiently large prime power
for which $d_0$ is the smallest prime divisor of $q-1$. If
\[
 (d_0-\tfrac12)\log q\le s\le q,
\]
then $G_{d_0}^0$ maximizes the number of covering $s$-element shift
sets among the sets $G_d^0$ with $d>1$ and $d\mid q-1$.
\end{proposition}
\begin{proof}
If $s>n_{d_0}$, then $N_{d_0}=0$. Otherwise, for every $d>d_0$,
the elementary counting bounds and the product formula give
\[
 \frac{N_{d_0}}{N_d}
 \le q\left(\frac{\alpha_{d_0}}{\alpha_d}\right)^s
 \le q(1-d_0^{-2})^s,
\]
where the last inequality uses $d\ge d_0+1$.
This proves the comparison when $q(1-d_0^{-2})^s\le1$.
In the remaining range, $s=O_{d_0}(\log q)$, so the Cauchy--Schwarz
lower bound above and $N_{d_0}/N\le q\alpha_{d_0}^s$ give
\[
 \frac{N_{d_0}}{N_d}
 \ll q\alpha_{d_0}^s+(1-d_0^{-2})^s.
\]
The midpoint bound for the integral of $1/t$ gives
$\log(d_0/(d_0-1))>1/(d_0-\tfrac12)$. Thus both terms are
$o_{d_0}(1)$, uniformly over the remaining $s$. For sufficiently
large $q$, this gives $N_{d_0}\le N_d$ for every $d>d_0$.
\end{proof}

\begin{remark*}
The comparison is natural, since $G_{d_0}^0$ has the largest cardinality
among these sets. If $d_0\mid d$, then $G_d^0\subseteq G_{d_0}^0$,
so the comparison holds for every $s$. In particular, no restriction on
$s$ is needed when $q-1$ is a prime power. For subgroups not related
by containment, comparison of means alone does not suffice.

The first part of the proof also gives a sufficient condition for
every $q>2$, allowing $d_0$ to vary:
\[
 s\ge\frac{\log q}{\log\bigl(d_0^2/(d_0^2-1)\bigr)}.
\]
\end{remark*}

For $s=O(\log q)$ and sufficiently large $q$, taking complements in
the Cauchy--Schwarz estimate bounds the proportion of covering shift sets:
\[
 1-\frac{N_d}{N}\ll\frac1{q\alpha_d^s}+\frac{s}{\sqrt q},
\]
uniformly in $d$. Higher binomial moments give a uniform formula
throughout the transition between covering and noncovering.

\begin{proposition}\label{prop:uniform-covering}
Fix $C>0$, and choose $A$ uniformly among the $s$-element subsets of
$\F_q$. Put $\lambda_d=q\alpha_d^s$.
For all sufficiently large prime powers $q$,
\begin{equation}\label{eq:uniform-covering-probability}
 \Pr(A+G_d^0=\F_q)=e^{-\lambda_d}+O_C(q^{-1/12}),
\end{equation}
uniformly for all integers $1\le s\le C\log q$ and all $d>1$ dividing
$q-1$. The lower bound on $q$ and the implied constant depend only on $C$.
\end{proposition}
\begin{proof}
We first estimate the binomial moments up to order
\[
 K=\lfloor(\log q)/3\rfloor.
\]

For each $r$-element set $X$, \eqref{eq:uncovered-centers} and
$\alpha_d\ge1/2$ give
\[
 U_d(X)=q\alpha_d^r\left(1+O\left(
 \frac{r4^r}{\sqrt q}\right)\right).
\]
The relative error multiplied by $s$ is $o_C(1)$ uniformly for
$r\le K$, since $(\log4)/3<1/2$. In particular,
$U_d(X)\ge\tfrac12q2^{-r}\gg s$ for sufficiently large $q$,
uniformly in the stated ranges. The product formula, including the
error from sampling without replacement, gives
\[
 \frac{\binom{U_d(X)}s}{\binom qs}
 =\prod_{j=0}^{s-1}\frac{U_d(X)-j}{q-j}
 =\alpha_d^{rs}\left(1+O_C\left(
 \frac{sr4^r}{\sqrt q}+\frac{s^22^r}{q}\right)\right).
\]
Since $s\le C\log q$, we have, uniformly for $1\le r\le K$,
\[
 \binom qr=\frac{q^r}{r!}\bigl(1+O(r^2/q)\bigr),\qquad
 \frac{s^22^r+r^2}{q}\ll_C\frac{sr4^r}{\sqrt q}.
\]
Thus dividing \eqref{eq:uncovered-double-count} by $\binom qs$ and
absorbing these smaller errors gives
\begin{equation}\label{eq:uncovered-binomial-moments}
 \mathbb E\binom{U_d(A)}r
 =\frac{\lambda_d^r}{r!}\left(1+O_C\left(
 \frac{sr4^r}{\sqrt q}\right)\right),\qquad 1\le r\le K.
\end{equation}

We next estimate the probability of no uncovered points. Put
\[
 L=\frac{\log q}{12},
\]
and first suppose that $\lambda_d\le L$.
For every integer $n\ge0$ and $0\le t\le1$, Taylor's theorem gives
\[
 \left|(1-t)^n-\sum_{r=0}^{K-1}(-t)^r\binom nr\right|
 \le t^K\binom nK.
\]
At $t=1$, combine this bound with the Taylor remainder for
$e^{-\lambda_d}$ and \eqref{eq:uncovered-binomial-moments}.
Extending the resulting positive error sum gives
\begin{align*}
 \left|\Pr(U_d(A)=0)-e^{-\lambda_d}\right|
 &\ll_C\frac{s}{\sqrt q}\sum_{r\ge1}\frac{r(4\lambda_d)^r}{r!}
       +\frac{\lambda_d^K}{K!}\\
 &\ll_C\frac{s\lambda_d e^{4\lambda_d}}{\sqrt q}
       +\left(\frac{\mathrm eL}{K}\right)^K
 =O_C(q^{-1/12}).
\end{align*}
Here the first term is $O_C((\log q)^2q^{-1/6})$, and the
last is $q^{-(\log4-1)/3+o(1)}=o(q^{-1/12})$.

If $\lambda_d>L$, apply the same Taylor bound with $t=L/\lambda_d$.
Each term of order $r$ is multiplied by $(L/\lambda_d)^r$, so the
preceding calculation now has parameter $L$. Hence
\[
 \Pr(U_d(A)=0)
 \le\mathbb E\left(1-\frac L{\lambda_d}\right)^{U_d(A)}
 =e^{-L}+O_C(q^{-1/12})=O_C(q^{-1/12}).
\]
Together with $e^{-\lambda_d}\le e^{-L}=q^{-1/12}$, this proves
\eqref{eq:uniform-covering-probability} in the remaining range.
\end{proof}

\begin{remark*}
For fixed $d>1$, put
\[
 c_d=\frac1{\log(d/(d-1))}.
\]
Since
$\lambda_d=\alpha_d^{\,s-c_d\log q}$,
Proposition~\ref{prop:uniform-covering} gives, throughout the critical
window $s=c_d\log q+O(1)$,
\[
 \Pr(A+G_d^0=\F_q)
 =\exp\left(-\alpha_d^{\,s-c_d\log q}\right)+O_d(q^{-1/12}).
\]
The formula retains the bounded displacement caused by integer rounding: choosing
$s=\lceil c_d\log q\rceil$ need not give a single limiting probability.
At a fixed positive distance above or below $c_d$ in $s/\log q$,
the covering probability tends to one or zero, respectively.
Since $d-1<c_d<d$, an integer $k$ lies strictly between $c_k$ and
$c_{k+1}$, which explains the following cutoff.
\end{remark*}

\begin{corollary}\label{cor:counting-threshold}
Fix an integer $k\ge2$. For all sufficiently large prime powers $q$
and integers $1\le s\le\lceil k\log q\rceil$, all but an
$O_k(q^{-\delta})$ proportion of the $s$-element subsets $A\subseteq\F_q$ satisfy
\[
 A+G_d^0=\F_q\quad\Longrightarrow\quad d\le k
 \qquad\text{for every }d>1\text{ dividing }q-1.
\]
Here $\delta>0$ depends only on $k$.
For $s=\lceil k\log q\rceil$, the implication can be replaced by an equivalence.
More precisely, one may take
\[
 \delta=\frac1{2\max\{6,k\}}.
\]
In particular, for odd $q$ and $s=\lceil2\log q\rceil$, the only covering index is $2$
apart from an $O(q^{-1/12})$ proportion of the shift sets.
\end{corollary}
\begin{proof}
Since $s\le\lceil k\log q\rceil$, Corollary~\ref{cor:index-growth}
leaves only $O_k(1)$ candidate indices, in a bounded range depending
only on $k$ and independent of $A$.
For each such index $d>k$,
\[
 \lambda_d=q\alpha_d^s
 \ge\frac{k}{k+1}q^{\,1-k\log(1+1/k)}.
\]
The exponent is positive because $k\log(1+1/k)<1$.
Proposition~\ref{prop:uniform-covering}, applied with $C=k+1$, therefore gives
\[
 \Pr(U_d(A)=0)=O_k(q^{-1/12}).
\]
Summing over the bounded range of candidate indices proves the
first assertion.

Now let $s=\lceil k\log q\rceil$. For $2\le d\le k$ with $d\mid q-1$,
we have
\[
 \lambda_d\le q^{\,1-k\log(k/(k-1))}.
\]
Since
\[
 k\log\frac{k}{k-1}-1
 =\int_0^1\frac{t}{k-t}\,dt>\frac1{2k},
\]
the same proposition and $1-e^{-\lambda_d}\le\lambda_d$ give
\[
 \Pr(U_d(A)>0)
 \le\lambda_d+O_k(q^{-1/12})
 =O_k(q^{-\delta}).
\]
Another finite union bound proves the equivalence.
For odd $q$, the index $2$ divides $q-1$, so taking $k=2$ gives
the final assertion.
\end{proof}

\section{Subfield shift sets and the projection bound}\label{sec:subfields}
We first describe decomposition over additive cosets and characterize
equality in the counting bound. We then take a subfield as the shift
set and seek coverings by smaller multiplicative subgroups than in
Example~\ref{prop:quadratic-covering}.

\begin{proposition}\label{prop:subfield-reduction}
Let $V$ be an additive subgroup of $\F_q$ containing $G_d^0$. Then
\[
 \tau(q,d)=\frac{q}{\#V}\min\{\#A:A\subseteq V,\ A+G_d^0=V\}.
\]
In particular, if $G_d$ is contained in the multiplicative group of a proper
subfield, every covering \eqref{eq:covering} satisfies $q\le s^2$.
\end{proposition}
\begin{proof}
Each translate lies in a single coset of $V$, and each coset requires
the same minimum number of translates. This proves the formula.
A proper subfield has at most $\sqrt q$ elements, so its additive
cosets require at least $\sqrt q$ translates.
\end{proof}

\begin{proposition}\label{prop:counting-equality}
For a covering \eqref{eq:covering}, equality
\[
 q=s\#G_d^0
\]
holds if and only if $G_d^0$ is a subfield of $\F_q$ and $A$ contains exactly
one representative of each of its additive cosets.
\end{proposition}
\begin{proof}
Write $q=p^n$. Equality gives $\#G_d^0=p^a$ for some $a\ge1$,
and $p^a-1$ divides $p^n-1$. The identity
\[
 \gcd(p^a-1,p^n-1)=p^{\gcd(a,n)}-1
\]
therefore gives $a\mid n$. By uniqueness of the multiplicative
subgroup of order $p^a-1$, we have $G_d^0=\F_{p^a}$.
Its translates are the additive cosets of this subfield, and equality
forces $A$ to contain exactly one representative of each.
The converse is immediate.
\end{proof}

For $d>s$, the elementary counting bound gives
\[
 q\le s\left(\frac{q-1}{d}+1\right)
 \quad\Longrightarrow\quad q\le\frac{s(d-1)}{d-s}\le s^2,
\]
where the last inequality uses $d\ge s+1$.
Thus $q>s^2$ forces $d\le s$.
If $q=s^2$ and $d>s$, equality forces $d=s+1$.
Proposition~\ref{prop:counting-equality} then gives $G_d^0=\F_s$,
with $A$ a complete set of additive coset representatives in $\F_{s^2}$.
Conversely, $\tau(s^2,s+1)=s$ for every prime power $s$.

We now use the subfield $\F_q$ as the shift set in $\F_{q^2}$,
where $q$ is a prime power. Throughout the rest of this section,
\[
 G_d=(\F_{q^2}^\times)^d.
\]
The $q-1$ nonzero additive cosets of $\F_q$ must all be met by
$G_d\setminus\F_q$. Counting these elements gives the following bound.
For character-sum estimates on power residues in individual translates
of subfields, see \cite[Theorem~1.3]{martinyip}.

\begin{proposition}\label{prop:quadratic-bound}
Suppose that $\F_q+G_d^0=\F_{q^2}$, where $d>1$ divides $q^2-1$.
Then
\begin{equation}\label{eq:quadratic-bound}
 d+\gcd(d,q+1)\le q+1.
\end{equation}
If equality holds, every $x\in\F_{q^2}\setminus\F_q$ has a unique
representation $x=a+h$ with $a\in\F_q$ and $h\in G_d$.
\end{proposition}
\begin{proof}
Choose a nonzero $\F_q$-linear map $\pi:\F_{q^2}\to\F_q$ with kernel $\F_q$.
Then
\begin{equation}\label{eq:projection}
 \F_q+G_d^0=\F_{q^2}\quad\Longleftrightarrow\quad\pi(G_d^0)=\F_q.
\end{equation}
By cyclicity,
\[
 \#(G_d\cap\F_q^\times)
 =\gcd\!\left(\frac{q^2-1}{d},q-1\right)
 =\frac{q-1}{d}\gcd(d,q+1).
\]
Every nonzero value of $\pi$ must be attained on $G_d\setminus\F_q$, so
\[
 q-1\le\#(G_d\setminus\F_q)
 =\frac{q-1}{d}\bigl(q+1-\gcd(d,q+1)\bigr).
\]
This proves \eqref{eq:quadratic-bound}.
If equality holds, $\pi$ maps $G_d\setminus\F_q$ bijectively onto
$\F_q^\times$. For $x\notin\F_q$, the unique $h\in G_d$ with
$\pi(h)=\pi(x)$ then gives $a=x-h\in\F_q$.
\end{proof}

For odd $q$, Example~\ref{prop:quadratic-covering} gives $d=(q+1)/2$.
The construction below uses cube roots of unity to give an exact
covering criterion at a larger index. The subgroup has smaller
cardinality and is not contained in the one from that example.

\begin{example}\label{prop:cubic-family}
Let $q$ be an odd prime power with $q\equiv2\pmod3$, and put $d=2(q+1)/3$.
Choose a primitive cube root of unity $\zeta\in\F_{q^2}$.
Since $3\nmid q-1$, we have $\zeta\notin\F_q$.
Put $T=(\F_q^\times)^2$. The group $T\langle\zeta\rangle$ has order
$3(q-1)/2$ and hence equals $G_d$.
In the basis $1,\zeta$, use the projection $\pi(x+y\zeta)=y$.
The identity $\zeta^2=-1-\zeta$ gives
\[
 \pi(G_d^0)=\{0\}\cup T\cup(-T).
\]
The image is $\F_q$ exactly when $-1$ is a nonsquare, so
\[
 \F_q+G_d^0=\F_{q^2}
 \quad\Longleftrightarrow\quad q\equiv3\pmod4.
\]
Otherwise the image has $(q+1)/2$ elements. Since each value of $\pi$
has $q$ preimages, the complement has $q(q-1)/2$ elements.

Taking $q=11^{2t+1}$, $t=0,1,\ldots$, gives an infinite covering family with
$d/q\to2/3$. For example, $\F_{11}+G_8^0=\F_{121}$,
whereas $\F_5+G_4^0$ misses $10$ elements of $\F_{25}$.

When the covering holds, $\gcd(d,q+1)=(q+1)/3$ gives equality in
\eqref{eq:quadratic-bound}. In particular,
\[
 \#\{(a,h)\in\F_q\times G_d^0:x=a+h\}
 =\begin{cases}
   1,&x\in\F_{q^2}\setminus\F_q,\\
   (q+1)/2,&x\in\F_q.
  \end{cases}
\]
Thus all overlaps lie in $\F_q$. The equality case in
Proposition~\ref{prop:counting-equality} instead has unique representations
throughout the field.
\end{example}

\end{document}